\documentclass[12pt,reqno]{amsart}
\usepackage{verbatim}
\usepackage{amssymb}
\usepackage{enumerate}
\usepackage[active]{srcltx}
\numberwithin{equation}{section}

\usepackage{t1enc}
\usepackage[utf8x]{inputenc}

\newtheorem{theorem}{Theorem}[section]
\newtheorem{lemma}[theorem]{Lemma}

\newtheorem{claim}{Claim}[theorem]
  
\newtheorem{proposition}[theorem]{Proposition}
\newtheorem{problem}[theorem]{Problem}

\theoremstyle{definition}
\newtheorem{definition}[theorem]{Definition}

\theoremstyle{remark}

\newcommand{\mc}[1]{\mathcal{#1}}

\newcommand{\mf}[1]{\mathfrak{#1}}

\newcommand{\setm}{\setminus}
\newcommand{\empt}{\emptyset}
\newcommand{\subs}{\subset}
\newcommand{\dom}{\operatorname{dom}}

\def\<{\left\langle}
\def\>{\right\rangle}
\def\br#1;#2;{\bigl[ {#1} \bigr]^ {#2} }

\newcommand{\pind}{\operatorname{pd}}
\newcommand{\den}{\operatorname{d}}

\DeclareMathOperator{\pp}{pp}
\DeclareMathOperator{\mDelta}{\Delta}
\DeclareMathOperator{\cf}{cf}
\DeclareMathOperator{\we}{w}

\newcommand{\inte}[1]{{\mathbb I}_{#1}}

\author[I. Juh\'asz]{Istv\'an Juh\'asz}
\address      { Alfréd Rényi Institute of Mathematics, Hungarian Academy of Sciences }
\email{juhasz@renyi.hu}

\author[L. Soukup]{Lajos Soukup}
\thanks
  {
   }
\address
      { Alfr{\'e}d R{\'e}nyi Institute of Mathematics, Hungarian Academy of Sciences }
\email{soukup@renyi.hu}

\author[Z. Szentmikl\'ossy]{Zolt\'an Szentmikl\'ossy}
\address{E\"otv\"os University of Budapest}
\email{szentmiklossyz@gmail.com
}

\title[Pinning Down vs Density]
   {Pinning Down versus Density}
\date{\today}
\thanks{The research on and preparation of this paper was
supported by  OTKA grant no. K113047.}

\begin{document}

\begin{abstract}
The {\em pinning down number}  $ \pind(X)$ of a topological space $X$ is
the smallest cardinal $\kappa$ such that for any
neighborhood assignment
$U:X\to \tau_X$ there is a set $A\in \br X;\kappa;$
with  $A\cap U(x)\ne\empt$
for all $x\in X$. Clearly, c$(X) \le \pind(X) \le \den(X)$.

Here we prove that
the following statements are equivalent:
\begin{enumerate}[(1)]
\item
$2^\kappa<\kappa^{+\omega}$
for each  cardinal $\kappa$;
\item $\den(X)=\pind(X)$ for each Hausdorff space $X$;
\item $\den(X)=\pind(X)$ for each 0-dimensional Hausdorff space $X$.
\end{enumerate}
This answers two questions of Banakh and  Ravsky.

The {\em dispersion character $\Delta(X)$} of  a space $X$
is the smallest cardinality of a non-empty open subset of $X$.
We also show that if $\pind(X)<\den(X)$ then $X$ has an open subspace $Y$
with $\pind(Y)<\den(Y)$ and $|Y| = \Delta(Y)$, moreover
the following three statements are {\em equiconsistent}:
\begin{enumerate}[(i)]
\item There is a singular cardinal $\lambda$ with $\pp(\lambda)>\lambda^+$,
i.e. Shelah's Strong Hypothesis fails;
\item there is a 0-dimensional Hausdorff space $X$ such that $|X|=\Delta(X)$
is a regular cardinal and $\pind(X)<\den(X)$;
\item there is a topological space $X$ such that $|X|=\Delta(X)$
is a regular cardinal and $\pind(X)<\den(X)$.
\end{enumerate}
We also prove   that
\begin{itemize}
\item $\den(X)=\pind(X)$ for any locally compact Hausdorff space $X$;
\smallskip
\item for every Hausdorff space $X$ we have $|X|\le 2^{2^{\pind(X)}}$ and
$\pind(X)<\den(X)$ implies $\Delta(X)< 2^{2^{\pind(X)}}$;
\smallskip
\item for every regular space $X$ we have $\min\{\Delta(X),\, \we(X)\}\le 2^{\pind(X)}\,$ and
$\den(X)<2^{\pind(X)},\,$ moreover
$\pind(X)<\den(X)$ implies $\,\Delta(X)< {2^{\pind(X)}}$.
\end{itemize}
\end{abstract}

\maketitle

\section{Introduction}

\begin{definition}
Let $X$ be a topological space.
We say that $A\subs X$ {\em pins down}
a neighborhood assignment
$U:X\to \tau_X$ iff $A\cap U(x)\ne\empt$
for all $x\in X$.
The {\em pinning down number}  $ \pind(X)$ of $X$ is the smallest cardinal $\kappa$ such that
every neighborhood assignment on $X$ can be pinned down by a set of size
$\kappa$.
\end{definition}
Clearly, for any space $X$ we have c$(X) \le \pind(X) \le \den(X)$.

\medskip

The pinning down number has been recently introduced in  \cite{Ba}
under the name ``{\em foredensity}'' and it was denoted there by $\ell^-(X)$.
The following two interesting results concerning the pinning down number were also established in \cite{Ba}:

\begin{itemize}
\item
\cite[Theorem 5.2]{Ba}
If $|X|<\aleph_\omega$ then $\pind(X)=\den(X)$.

\smallskip

\item
\cite[Corollary 5.4]{Ba}
If ${\kappa}$ is any singular cardinal then
there is a $T_1$ semitopological group $X$ such that $$\pind(X)=\cf({\kappa})
< \kappa = \den(X)=|X|= \Delta(X).$$ Moreover, if ${\kappa} < 2^{2^{\cf({\kappa})}}$ then
$X$ is even Hausdorff and totally disconnected.
\end{itemize}

The following two natural problems were then raised in \cite{Ba}:

\begin{itemize}
\item
\cite[Problem 5.5]{Ba}
Is there a ZFC example of a Hausdorff space $X$ with
$\pind(X) < \den(X)$?
\item
\cite[Problem 5.6]{Ba}
Is it consistent to have a regular space $X$ with $\pind(X) < \den(X)$?
\end{itemize}

\medskip
Our next result completely settles both of these problems.

\begin{theorem}\label{tm:main1}
The following three statements are equivalent:
\begin{enumerate}[(1)]
\item $2^\kappa<\kappa^{+\omega}$ for each
cardinal  $\kappa$;
\item $\den(X)=\pind(X)$ for every Hausdorff space $X$;
\item $\den(X)=\pind(X)$ for every 0-dimensional Hausdorff space $X$.
\end{enumerate}
\end{theorem}

We shall say that a topological space $X$ is {\em neat} iff
$X \ne \empt$ and $|X|=\Delta(X)$,
where the {\em dispersion character $\Delta(X)$} of $X$
is the smallest cardinality of a non-empty open subset of $X$.
In other words, $X$ is neat iff all non-empty open sets in $X$ have the same size.
We shall show in the next section that any space $X$ satisfying $\pind(X)<\den(X)$
has a neat open subspace $Y$ with $\pind(Y)<\den(Y)$.

The examples that Banakh and
Ravsky constructed in the proof of \cite[Corollary 5.4]{Ba},
as well as the examples we first constructed in our proof of theorem \ref{tm:main1}
were both neat and of singular cardinality. Hence it was natural for us to raise the
question if witnesses for $\pind(X)<\den(X)$ that are both neat and of regular cardinality
could also be found.

Before discussing our answer to this question,
we need to recall {\bf Shelah's Strong Hypothesis} which is the following statement:
\begin{align}\label{SSH}
\text{$\pp(\mu)=\mu^+$ for all singular cardinals $\mu$}.
\end{align}

Our next result gives an answer to the previous question that is complete up to consistency.

\begin{theorem}\label{tm:equicons}
The following statements
are equiconsistent:
\begin{enumerate}[(i)]
\item Shelah's Strong Hypothesis fails;

\item there is a neat 0-dimensional Hausdorff space $X$ of regular
cardinality with $\pind(X)<\den(X)$;

\item there is a neat topological space $X$ of regular cardinality with $\pind(X)<\den(X)$.
\end{enumerate}
\end{theorem}

We shall prove both theorems \ref{tm:main1} and \ref{tm:equicons} in section \ref{sc:pind}.

\medskip

In the last section of the paper we shall establish several interesting
inequalities involving the pinning down number. Perhaps the most interesting and
surprising of these is theorem \ref{tm:cardless22pdX} which states that
$|X|\le 2^{2^{\pind(X)}}$ holds
for every Hausdorff space $X$. This, of course, improves Pospi\v sil's classical
inequality $|X|\le 2^{2^{\den(X)}}$.

\section{Preliminary results}

In this section we present several rather simple results that, however, will be frequently used
in the proofs of our main results. We start with a proposition that describes the monotonicity
properties of $\pind(X)$. These are so obvious that we omit their proofs.

\begin{proposition}\label{prop:im}
\begin{enumerate}[(i)]
\item If $G$ is an open subspace of  $X$ then $\pind(G) \le \pind(X)$;

\item  if $f : X \to Y$ is a continuous onto map then $\pind(Y) \le \pind(X)$.
\end{enumerate}
\end{proposition}

We now give the result that was promised in the introduction.

\begin{lemma}\label{lm:ex2cardhomex1}
If $\pind(X)<\den(X)$ then $X$ has a neat open subspace $Y$
with $\pind(Y)<\den(Y)$.
\end{lemma}

\begin{proof}[Proof of Lemma  \ref{lm:ex2cardhomex1}]
Clearly, every non-empty open set in $X$ has a neat open subset, hence if
$\mc U$ is a maximal family of pairwise disjoint neat
open subsets   of $X$
then $\bigcup \mc U$ is dense open in $X$ and, consequently, $\den(\bigcup \mc U) = \den(X)$.
Let us put
\begin{align}\notag
\mathcal{V} = \{U\in \mc U: \den(U)\le \pind(X)\},
\end{align}
then $|\mathcal{V}| \le c(X)\le \pind(X)$ implies
$\den(\cup \mathcal{V})\le \pind(X)<\den(X)=\den(\bigcup \mc U)$,
and so $\mathcal{V} \ne \mathcal{U}$.
But every $Y\in \mc U \setm \mathcal{V}$ is neat open and, by definition, satisfies $\den(Y)>\pind(X)\ge \pind(Y)$.
\end{proof}

The basic idea of the following lemma goes back to \cite{Ba}.

\begin{lemma}\label{lm:cfkappalambda1}
Assume that $\lambda\le |X|=\Delta(X)=\kappa$.
If there is a family $\mc A\subs \br {\kappa};<\den(X);$ with
$|\mc A|={\kappa}$ such that
 \begin{align}\notag
\br {\kappa};<{\lambda};=\bigcup_{{A}\in \mc A}\br A;<{\lambda};
 \end{align}
then $\,\pind(X)\ge\lambda$. In particular, if for every cardinal $\mu < \den(X)$ we have
$\cf([\kappa]^\mu, \subs) = \kappa$ then $\pind(X) = \den(X)$.
\end{lemma}

\begin{proof}[Proof of Lemma  \ref{lm:cfkappalambda1}]
We may assume that the underlying set of $X$ is $\kappa$.
Write $\mc A=\{A_\nu:\nu<\kappa\}$ and,
by transfinite recursion, pick points
$\{x_\nu:\nu<\kappa\}$ from $X$ such that for each $\nu < \kappa$
\begin{align}\notag
 x_\nu\in (X\setm \overline{A_\nu})\setm \{x_\mu:\mu<\nu\}.
\end{align}
This can be done because $A_\nu$ is not dense in $X$, hence
$|X\setm \overline{A_\nu}|=\kappa$.

Let $U$ be a neighborhood assignment of $X$ such that
\begin{align}\notag
U(x_\nu)=X\setm \overline{A_\nu}
\end{align}
for all $\nu<\kappa$.
For every $D\in \br X;<\lambda;$ then, by our assumption, there is $\nu<\kappa$ with $D\subs A_\nu$,
hence $D\cap U(x_\nu)=\empt$, i.e.
$D$ does not pin down
$U$. Consequently, we indeed have $\pind(X) \ge \lambda$.

The second statement follows by applying the first one with $\lambda = \mu^+$ for all $\mu < \den(X)$.
\end{proof}

It is well-known that for every infinite cardinal $\kappa < \aleph_\omega$ we have
$\cf\big([\kappa]^{<\kappa}, \subs \big) = \kappa$, so we can easily deduce
from the previous two lemmas that $|X| < \aleph_\omega$ implies $\pind(X)=\den(X)$.
Our next two results give further ways to deduce this equality.

\begin{lemma}\label{lm:piDelta}
If  $X$ satisfies $\Delta(X)\ge\pi(X)$ then $\pind(X)=\den(X)$.
\end{lemma}

\begin{proof}
Write $\kappa=\pi(X)$ and $\mathcal{P} = \{U_\nu:\nu<\kappa\}$ be a $\pi$-base of $X$.
By transfinite recursion we may then pick points
$\{x_\nu:\nu<\kappa\}$ from $X$ such that for each $\nu < \kappa$
\begin{align}\notag
 x_\nu\in U_\nu \setm {\{x_\mu:\mu<\nu\}}.
\end{align}
This is possible because $|U_\nu| \ge \Delta(X) \ge \kappa$.

Let $U$ be a neighborhood assignment on $X$ such that
\begin{align}\notag
 U(x_\nu)=U_\nu \end{align}
holds for all $\nu<\kappa$. Then any set that pins down $U$ meets
every member of $\mathcal{P}$, and so is dense in $X$,
hence $\pind(X)=\den(X)$.
\end{proof}

\begin{lemma}\label{lm:pig}
If $X$ is any topological space and
\begin{align}\notag
\mc G=\{G\in \tau_X: \pi(G)\le |G|\}
\end{align}
is a $\pi$-base of $X$ then $\pind(X) = \den(X)$.
\end{lemma}

\begin{proof}
 Clearly, if
 $\mc G$ is a $\pi$-base of $X$ then so is
 \begin{align}\notag
 \mc H=\{G\in \mc G: |G|=\Delta(G)\}
 \end{align}
and, by Lemma \ref{lm:piDelta}, we have
\begin{align}\notag
 \pind(G)=\den(G)
\end{align}
for all $G\in \mc H$.

Let $\mc U$ be a maximal family of pairwise disjoint
elements of $\mc H$.
Then $\bigcup \mc U$ is dense open in $X$ and
$|\mc U|\le \operatorname{c}(X)\le \pind(X)$.
So we have
\begin{align}\notag
 \den(X) = \den(\bigcup \mc U) = \sum_{U\in \mc U}\den(U)=
\sum_{U\in \mc U}\pind(U)\le |\mathcal{U}|\cdot\pind(X)=\pind(X),
\end{align}
and hence $\pind(X) = \den(X)$.
\end{proof}

As a corollary of this we get the following result.

\begin{theorem}\label{tm:compact}
For every locally compact Hausdorff space $X$ we have $\pind(X)=\den(X)$.
\end{theorem}

\begin{proof}
By lemma \ref{lm:pig} it suffices to show that
\begin{align}\notag
\{G\in \tau_X: \pi(G)\le |G|\}
\end{align}
is a $\pi$-base of $X$.

But it is well-known that even the weight of a  locally compact Hausdorff space
is less than or equal to its cardinality, hence we have $\pi(G) \le |G|$ for
all non-empty open sets $G$ in $X$.
\end{proof}

It is, of course, a natural question to raise if this equality holds for the members
of other classes of spaces. In particular, we could not answer the following
questions.

\begin{problem}
Does $\pind(X)=\den(X)$ hold true if $X$ is
\begin{enumerate}[(i)]
\item regular $\sigma$-compact, or
\item regular Lindel{\"o}f, or
\item regular countably compact, or
\item monotonically normal ?
\end{enumerate}
\end{problem}

\section{The pinning down number and cardinal arithmetic}\label{sc:pind}

\bigskip

Our first result in this section establishes the implication $(1) \Rightarrow (2)$
in theorem \ref{tm:main1}.

\begin{theorem}\label{tm:neg1}
If  $X$ is any Hausdorff space
with $$\mu \le |X|=\Delta(X) < \mu^{+\omega}$$ where $\mu$ is strong limit
then $\den(X)=\pind(X)$.
\end{theorem}

\begin{proof}[Proof of Theorem \ref{tm:neg1}]
Since $\mu\le |X|\le 2^{2^{\den(X)}}$ and $\mu $ is strong limit,
we have $\den(X)\ge \mu$. Now we distinguish two cases.

\noindent{\bf Case 1.}{ $\den(X)=\mu$.}

Instead of our space $(X,\tau)$ we may take
a  coarser Hausdorff  topology $\sigma$ on $X$ such that for the space space $X^*=(X,\sigma)$
we have  $\we(X^*)\le |X| = |X^*|$.
Clearly, we also have $\pind(X^*)\le \pind(X)$.
Since $\mu$ is strong limit and $X^*$ is Hausdorff,  $\den(X^*) = \mu$ holds as well.

We also have $\Delta(X^*) = \Delta(X)=|X|=|X^*|\ge \we(X^*) \ge \pi(X^*)$,  hence
by Lemma \ref{lm:piDelta}, $\den(X^*)=\pind(X^*)$.
So we have $\mu = \den(X^*)=\pind(X^*)\le \pind(X)\le \den(X)=\mu$,
which completes the proof in this case.

\medskip
\noindent{\bf Case 2.}{ $\den(X)>\mu$.}

Then $\den(X)=\lambda^+$ for some cardinal $\lambda \ge \mu$
and $|X|=\lambda^{+m}$ for some $0 < m < \omega$.
But then we have $\cf(\br \lambda^{+m};\lambda;, \subs)=\lambda^{+m}$ and so
Lemma \ref{lm:cfkappalambda1} may be applied to conclude
$\pind(X)\ge \lambda^+=\den(X)$.
\end{proof}

In order to establish the implication $(3) \Rightarrow (1)$
in theorem \ref{tm:main1} we clearly need to show how to construct a
0-dimensional Hausdorff space $X$ satisfying $\pind(X) < \den(X)$
from the assumption that $2^\kappa > \kappa^{+\omega}$ for some cardinal $\kappa$.
Note that in this case $\kappa^{+\omega}$ is a singular cardinal that is not strong limit.
In fact, our construction may be carried out
for any singular cardinal that is not strong limit.

Actually, we shall introduce two extra parameters $\sigma$ and $\varrho$ in the construction which are
not needed just for the proof of theorem \ref{tm:main1}. The role of $\sigma$ is to show a great deal of
flexibility in the choice of the density of the space we construct, while $\varrho$
will be used in the proof of theorem \ref{tm:equicons},

Before formulating our result we first present Shelah's definition of
the ``pseudopower'' $\pp(\mu)$ of an arbitrary singular cardinal $\mu$.
This will be necessary to understand our construction.

In what follows,
$\mathfrak{Reg}$ denotes the class of regular cardinals. For a singular cardinal $\mu$
we let $$\mathcal{S}(\mu) = \{\mathfrak{a} \in [\mu \cap \mathfrak{Reg}]^{\cf(\mu)} :\, \sup \mathfrak{a} = \mu \}$$
and, for $\mathfrak{a} \in \mathcal{S}(\mathfrak{a})$,

\begin{align}\notag
\mathcal{U}(\mathfrak{a}) = \text{\{$D$ : $D$ is an  ultrafilter on $\mf a$ with}\, D\cap J^{bd}[\mathfrak{a}]=\empt \},
\end{align}
where $J^{bd}[\mathfrak{a}]$ denotes the ideal of bounded subsets of $\mathfrak{a}$.
The {\em pseudopower} $\pp(\mu)$ of a singular cardinal $\mu$ is now defined as follows (see e.g. \cite{AM}).

\begin{definition}
If $\mu$ is any singular cardinal then
\begin{align}\notag
\pp(\mu)=\sup \Big\{\cf(\prod \mf a/D) : \,\, \mf a \in \mathcal{S}(\mu) \mbox{ and }
D \in \mathcal{U}(\mathfrak{a})  \Big\}.
\end{align}
\end{definition}
It will be useful to give the following, obviously equivalent, reformulation of this.
 \begin{align}\notag
 \pp(\mu)=\notag
\sup\Big\{\cf(\prod_{i\in \cf(\mu)}& k(i)/D):k\in {}^{\cf(\mu)}
{(\mu \cap \mf{Reg})}\,\mbox{ and}
\\\notag & \text{ $D$ is an ultrafilter on $\cf(\mu)$ with
$\lim\nolimits_D k=\mu$
}
 \Big\},\notag
 \end{align}
where $\lim\nolimits_D k=\mu$ means that
$\{i<\cf(\mu): k(i)>\nu\}\in D$  whenever $\nu<\mu$.

\medskip

Now, our desired construction in its most general form can be formulated as follows.

\begin{theorem}\label{tm:posplus}
Assume that $\mu$, $\lambda$, $\sigma$, and $\rho$
are infinite cardinals
such that
\begin{align}
\cf(\mu)\le \lambda <\sigma\le \mu\le \rho < \pp(\mu)\le 2^\lambda,
\end{align}
moreover
\begin{align}
  \sigma=\cf(\sigma) \text{ if } \sigma < \mu.
\end{align}
Then
there is a 0-dimensional Hausdorff space $X$
such that
\begin{enumerate}[(1)]
 \item $\pind(X)\le \lambda$,
\item $\den(X)=\sigma$,
\item $\Delta(X)=|X|=\rho$.
\end{enumerate}
In particular, if $\mu$ is a singular cardinal that is not strong limit then
there is a neat 0-dimensional Hausdorff space $X$ of size $\mu$ satisfying
$\pind(X) < \den(X) = \mu$.
\end{theorem}

\begin{proof}
It is easy to see from the above definition of $\pp(\mu)$ that, by $\rho<\pp(\mu)$,
there exist a regular cardinal $\kappa$ with $$\rho<\kappa\le \pp(\mu),$$
a function $k:\cf(\mu)\to \mu \cap \mf {Reg}$, and an
ultrafilter $D$
on $\cf (\lambda)$ with
$\lim\nolimits_Dk={\mu}$ such that
$$\cf(\prod\nolimits_{i\in \cf(\mu)} k(i)/D)=\kappa\,.$$
Since ${\lambda}<{\mu}$, we can assume without loss of generality that
\begin{align*}
 k(i)>{\lambda}\text{ for all $i<\cf({\mu})$.}
\end{align*}

Next we define two functions $k_1$ and $k_2$ with domain $\cf(\mu)$
as follows:
For any $i<\cf(\mu)$ we set
\begin{align}\notag
   k_1(i)=\begin{cases}
                 \sigma&\text{if $\sigma<\mu$,}\\
                 k(i)&\text{if $\sigma=\mu$}
                \end{cases}
  \end{align}
and
  \begin{align}\notag
   k_2(i)=\begin{cases}
                 \rho\cdot \mu^+&\text{if $\rho>\mu$,}\\
                 k(i)&\text{if $\rho=\mu$};
                \end{cases}
  \end{align}
here and in the rest of the proof ``$\,\cdot\,$'' always denotes ordinal multiplication. Hence in the case $\rho>\mu$
the values of $k_2$ are ordinals of size $\varrho$ that are not cardinals.
To simplify the notation we put
\begin{align*}
 k_0=k.
\end{align*}

Now, for each $m<3$ let us put
\begin{align}\notag
X_m=
\{\<i,m,{\alpha}\>\ :\  i<\cf({\mu}) \mbox{ and } {\alpha}<k_m(i)\}.
 \end{align}
The underlying set of  our space will be
\begin{align}\notag
X=\bigcup\nolimits_{m<3}X_m.
 \end{align}
Clearly this is a disjoint union and $|X_0| = \mu,\,|X_1| = \sigma,\,|X_2| = \varrho$, hence $|X| = \varrho$ as well.

Let us next put ${\kappa}_0={\kappa}$,
  \begin{align}\notag
\kappa_1=\begin{cases}
                 \sigma&\text{if $\sigma<\mu$,}\\
                 \kappa&\text{if $\sigma=\mu,$}
                \end{cases}
\end{align}
and
  \begin{align}\notag
   \kappa_2=\begin{cases}
                 \mu^+&\text{if $\rho>\mu$,}\\
                 \kappa&\text{if $\rho=\mu.$}
                \end{cases}
  \end{align}

Then for every $m<3$ we have $\cf(\prod_{i\in \cf(\mu)} k_m(i)/D) = \kappa_m$,
hence we may fix a $\le_D$-cofinal subfamily
$\mc F_m\subs \prod_{i\in \cf(\mu)} k_m(i)$
of cardinality $\kappa_m$.
Then we put $$\mc F=\mc F_0\times \mc F_1\times \mc F_2 ,$$
clearly, $\mathcal{F}$ has cardinality $\kappa$. Thus
every member $f\in \mc F$ is a triple of the form $f=\<f_0,f_1, f_2\>$
with $f_m \in \mathcal{F}_m$ for $m < 3$.
$\mathcal{F}$ will be used in the definition of the topology on $X$.

Next we fix an independent family
$\mc A \subs \br \lambda;\lambda;$ of cardinality $2^\lambda$.
Since $2^{\lambda}\ge \mu^{\cf(\mu)} \ge |X \times  \mathcal{F} \times D|$, we can also fix
an injection $$A : X\times \mc F\times D
\to   \mc A\,,$$
moreover we shall use the notation
\begin{align}\notag
 &A_0(x,f,d)=A(x,f,d) \text{ and }
A_1(x,f,d)=\lambda\setm A(x,f,d).
\end{align}
So, the injectivity of the map $A$ and the independence of $\mathcal{A}$ imply that for
every finite function $s \in Fn(X\times \mc F\times D,\,2 )$ we have

\begin{align}\notag
A_s =^{df} \bigcap_{(x,f,d)\in \dom  s}A_{s(x,f,d)}(x,f,d,) \ne \empt .
\end{align}

For any $x=(i,m,\zeta)\in X$ and $S\subs {\lambda}$ we shall write
\begin{align}\notag
x\oplus S=\{(i,m,\zeta\dotplus\eta): \eta\in S\},
\end{align}
where $\dotplus$ denotes ordinal addition.

Next, for any $x\in X$, $f\in \mc F$, and $d\in D$
we put
\begin{align}\label{eq:base}
B_0(x,f,d,)=\{x\}\cup \bigcup
\Big\{(j,m,\,&\lambda\cdot \alpha)\oplus A(x,f,d):\\
\notag
  &j\in d, m\in 3, f_m(j)<\alpha< k_m(j) \Big\}
\end{align}
and $B_1(x,f,d)=X\setm B_0(x,f,d)$.

For  $s\in Fn(X\times\mc F\times D, 2)$
let
\begin{align}\notag
B_s=\bigcap_{(x,f,d)\in \dom  s}B_{s(x,f,d)}(x,f,d).
\end{align}
Now, the family
\begin{align}\notag
 \mc B=\{B_s: s\in Fn(X\times\mc F\times D, 2)\}
\end{align}
will be the, obviously  clopen, base of our topology $\tau$ on $X$.

$\<X,\tau\>$ is Hausdorff because
if $x=\<i,m,{\alpha}\> \in X$ and $y \in X \setm \{x\}$ then for $d=\lambda\setm \{i\}\in D$
and an arbitrary $f\in \mc F$
we have $y\in B_0(y,f,d)$ but $x \notin B_0(y,f,d)$.

\medskip

The following observation will be crucial in the rest of our proof.
To simplify its formulation, we introduce the following piece of notation:

\begin{align*}
 \inte {\alpha}=
\big[{\lambda}\cdot {\alpha}, {\lambda}\cdot ({\alpha}\dotplus1)\ \big),
\end{align*}
where $\alpha$ is any ordinal. That is, $\inte {\alpha}$ is the interval
of order type $\lambda$ starting with ${\lambda}\cdot {\alpha}$.

\begin{claim}\label{cl:openlarge1}
Fix $s\in Fn(X\times\mc F\times D, 2)$ and assume that
$m\in 3$,   $i<\cf(\mu)$, and
$\alpha< k_m(i)$ are chosen in such a way that

\begin{align}\notag
\text{$i\in d$ and $\alpha > f_m(i)$
whenever $(x,f,d)\in \dom (s)$.}
\end{align}
Then

 \begin{align}\notag
 (\{i\}\times\{m\}\times\inte {\alpha})
\cap
 B_s\ne\empt.
 \end{align}
\end{claim}

\begin{proof}[Proof of the Claim]
Recall first that the set
\begin{align}\notag
A_s = \bigcap_{(x,f,d)\in \dom  s}A_{s(x,f,d)}(x,f,d,)
\end{align}
is non-empty. But if $\eta \in A_s$ then for every
$(x,f,d)\in \dom  s$ we  have
\begin{align}\notag
(i,m, \alpha \cdot \lambda \dotplus \eta)  \in
\big(\{i\}\times\{m\}\times
\inte {\alpha}\big)
\cap  B_{s(x,f,d)}(x,f,d)
\end{align}
because $i\in d$  and  $f_m(i)<\alpha$, hence
\begin{align}\notag
(i,m, \alpha \cdot \lambda \dotplus \eta)  \in
\big(\{i\}\times\{m\}\times
\inte {\alpha}\big)
\cap  B_s\,,
\end{align}
and this completes the proof.
\end{proof}

\begin{claim}
$\den(X)=\sigma$.
\end{claim}

\begin{proof}[Proof of the Claim]
For every basic clopen set $B_s \in \mathcal{B}$ we can
pick $i<\cf({\mu})$ and ${\alpha}<k_1(i)$ such that
$i\in d$
and  $f_1(i)<\alpha$ for all $(x,f,d)\in \dom  s$.
By Claim \ref{cl:openlarge1}
then we have
\begin{align}\notag
\big (\{i\}\times\{1\}\times
\inte {\alpha}
\big)
\cap B_s\ne \empt,
\end{align}
and so $X_1$
is dense in $X$. Consequently, $\den(X)\le |X_1| = \sigma$.

Now, consider an arbitrary set  $S \in \br X;<\sigma;$.
Then, of course,
 \begin{align}\notag
 d=
\{i\in \cf(\mu): k_0(i)>|S|\}
\in D.
 \end{align}
But $k_0(i)$ is regular for all $i$, hence
we can choose a function $p_0 \in \prod_{i\in \cf(\mu)} k_0(i)$
such that
\begin{align}\notag
S\cap (\{i\}\times \{0\}\times k_0(i))  \subs
\{i\}\times \{0\}\times\lambda\cdot p_0(i)
\end{align}
whenever $i \in d$.
We may then  pick $f\in \mc F$ such that $p_0\le_D f_0$.
Then we also have
\begin{align*}
e=\{i\in d : p_0(i)\le f_0(i)  \}\in D.
\end{align*}
But for any $i\in e$ and
$x\in \{i\}\times \{0\}\times \big(k_0(i) \setm
\lambda\cdot f(k)\big)$ we have
then $B_0(x,f,e)\cap S=\empt$, hence $S$ is not dense. Consequently, we indeed have $\,\den(X) = \sigma$.
\end{proof}

\begin{claim}
$\Delta(X)=\rho$.
\end{claim}

\begin{proof}[Proof of the Claim]
We know that $|X|=\rho$.
Now let $B_s \in \mathcal{B}$
be any basic open set.
Let us put
\begin{align}\notag
e=^{df}\bigcap\{d\in D: (x,f,d)\in \dom  s\}\in D.
\end{align}
Then, by Claim \ref{cl:openlarge1}, for every $i \in e$ and
for all  ${\alpha}$  with  $f_2(i)<\alpha<k_2(i)$ we have
\begin{align}\notag
\inte {\alpha}
\cap B_s\ne \empt,
\end{align}
and so
\begin{align}\label{eq:x2b}
|X_{2}\cap  B_s|\ge |k_2(i)\setm f_2(i)|=|k_2(i)|.
\end{align}

If $\rho>\mu$, then $|k_2(i)|=\rho$, hence $|B_s| = \rho$.
If $\rho=\mu$ then, as
\eqref{eq:x2b} holds for all $i\in e$, we have
\begin{align}\notag
|X_2\cap B_s|=
\sup_{i\in e}k(i)=\mu= \rho,
\end{align}
and so we conclude $|B_s| = \rho$ again.
Thus, indeed, we have $\Delta(X) = \rho$.
\end{proof}

\begin{claim}
$\pind(X)\le {\lambda}$.
\end{claim}

\begin{proof}[Proof of the Claim]
Clearly, it suffices to show that any neighborhood assignment of the form
$$\mathbb B =\<B_{s(y)}:y\in X\>$$ can be pinned down by a set of size $\lambda$, where $$s:X\to Fn(X\times \mc F\times D,2)$$
and $y \in B_{s((y)}$ for all $y \in X$.

Let us put
\begin{align}
 \mc F'=\{f \in \mathcal{F}:\,\exists\, (x,f,d)\in \dom (s(y))
\text{ for some } y\in X \}.
\end{align}
Then $|\mc F'|\le \rho<\kappa$ implies that there is a map $g\in \mc F_0$ such that
 \begin{align}
  f_0
\le_D g
 \end{align}
for all $f\in \mc F'$.

For every $i<\cf(\mu)$ let
\begin{align}\notag
J_i=\{i\}\times\{0\}\times \inte{g(i)}
\end{align}
and put
\begin{align}\notag
J=\bigcup_{i<\cf(\mu)}J_i.
\end{align}

Then $|J| = \lambda$ and we claim that $J$ pins down $\mathbb B$. To see this, let us
fix any $y\in X$ and set
\begin{align}\notag
e=\{i\in \cf({\mu}): i\in d \text{ and  }& f_0(i)\le g(i)
\\\notag&\text{ for all  }   (x,f,d)\in \dom  s(y)\}.
\end{align}
Then $e \in D$ and for any $i\in e$
we can
apply Claim \ref{cl:openlarge1} for $s(y)$, $0$, $i$ and $\alpha=g(i)$
to conclude that $J_i\cap B_{s(y)}\ne \empt$.
Thus, $J$ indeed pins down $\mathbb B$, which completes the proof.
\end{proof}

With this the proof of Theorem \ref{tm:posplus} has also been completed.
\end{proof}

Now we have more than necessary to prove theorem \ref{tm:main1}.

\begin{proof}[Proof of theorem \ref{tm:main1}]\makebox[1cm]{}\\
{\bf (1) implies (2)} is an immediate consequence of
theorem \ref{tm:neg1} and
lemma \ref{lm:ex2cardhomex1}.

\medskip

\noindent {\bf (2) implies (3)} is trivial.

\medskip

\noindent {\bf (3) implies (1)}. This, or rather its contrapositive,
follows immediately from theorem \ref{tm:posplus} because if $2^\kappa > \kappa^{+\omega}$
then $\mu = \kappa^{+\omega}$ is a singular cardinal that is not strong limit.
\end{proof}

Next we turn to the proof of theorem \ref{tm:equicons}. First we present a purely set-theoretic
statement, without proof, that is folklore and easy to prove.

\begin{proposition}\label{pr:cf}
If $\kappa$ is a regular cardinal and $\lambda < \kappa$ is such that
$\cf([\kappa]^\lambda, \subs) > \kappa$ then we have $\cf([\mu]^\lambda, \subs) > \mu^+$
for some singular cardinal $\mu < \kappa$.
\end{proposition}

>From this proposition and from lemma \ref{lm:cfkappalambda1} we can immediately
deduce the following result.

\begin{theorem}\label{tm:H2cov}
Assume that $X$ is any topological space for which $|X| = \Delta(X)$
is a regular cardinal and $\pind(X) < \den(X)$. Then there are a cardinal $\,\lambda < \den(X)$
and a singular cardinal $\mu < |X|$ such that $$\cf([\mu]^\lambda, \subs) > \mu^+\,.$$
\end{theorem}

But by \cite[Lemma 8.2]{Sh400a}, a highly non-trivial result of Shelah,
the existence of
a singular cardinal $\mu$ such that $\cf([\mu]^\lambda, \subs) > \mu^+$
for some $\lambda$ implies that SSH fails. Consequently, we have actually
established above the validity of the implication  (iii) $\Rightarrow$ (i)
in  theorem \ref{tm:equicons}. Since (ii) $\Rightarrow$ (iii) is trivial,
to complete the proof of theorem \ref{tm:equicons}
it only remains to show that Con(i) $\Rightarrow$ Con(ii).

Before doing that, however, we need the following lemma which is probably known.
Still we give its proof because we did not find any reference for it.
\begin{lemma}\label{lm:pp-preservation}
Assume that ${\mu}$ and ${\nu}$ are cardinals such that
\begin{align}\notag
2^{\cf({\mu})}<{\nu}<{\mu}.
\end{align}
Assume also that $W$ is an extension of our ground model $V$ such that
\begin{enumerate}
\item $On^W = On$ and $\alpha \le 2^{\cf(\mu)}$ implies $\cf^W(\alpha)= \cf(\alpha)$;
\smallskip
\item $W \vDash [V]^{2^{\cf(\mu)}} \subs V$;
\smallskip
\item
$W \vDash$  ``if $A \subs V$ and $|A| \ge \nu$ then there is $B \in V$
such that $A \subs B$ and $|A| = |B|$''.
\end{enumerate}
Then $\mu$ remains a singular cardinal in $W$, $(\mu^+)^W = \mu^+$, and
\begin{align}\label{eq:*}
\pp^W({\mu}) = \pp({\mu}).
\end{align}
Consequently, the failure of SSH in $V$ is preserved in $W$.
\end{lemma}

\begin{proof}
Only \eqref{eq:*} needs verification. To this end, note first that, by (3),
we have $\cf^W(\alpha)= \cf(\alpha)$ for any ordinal $\alpha$
such that $\cf^W(\alpha) \ge \nu$. This clearly implies that

\begin{align}\label{eq:regular}
 \mf{Reg}^W\setm {\nu}^+ =
 \mf{Reg}\setm {\nu}^+.
\end{align}

It follows from (2) that we also have
$$\mathcal{S}^W(\mu) \cap [\mu \setm \nu^+]^{\cf(\mu)} = \mathcal{S}(\mu) \cap [\mu \setm \nu^+]^{\cf(\mu)}\,.$$
Then, by (2) again, we clearly have
$$\mathcal{U}^W(\mathfrak{a}) = \mathcal{U}(\mathfrak{a}) \mbox{ and } \big(\prod \mf a\big)^W = \prod \mf a$$
whenever $\mathfrak{a} \in \mathcal{S}(\mu) \cap [\mu \setm \nu^+]^{\cf(\mu)}$.

Consequently, \eqref{eq:*} will follow if we can show that
$$\cf^W(\prod \mathfrak{a},\,\le_D) = \cf(\prod \mathfrak{a},\,\le_D)$$
whenever $\mathfrak{a} \in \mathcal{S}(\mu) \cap [\mu \setm \nu^+]^{\cf(\mu)}$ and
$D \in \mathcal{U}(\mathfrak{a})$.
To see this, let us fix, in $W$, any such $\mf a$ and $D$, moreover consider any $\le_D$-cofinal subset $A \subs \prod\mf a$.
Then $|A| > \mu > \nu$ implies by (3) that there is  $B \subs \prod \mf a$ such that
$B \in V$, $|A| = |B|$, and $A \subs B$. But then $B$ is also $\le_D$-cofinal in $\prod\mf a$,
which clearly implies that $\cf^W(\prod \mathfrak{a},\,\le_D) \ge \cf(\prod \mathfrak{a},\,\le_D)$.
But $\cf^W(\prod \mathfrak{a},\,\le_D) \le \cf(\prod \mathfrak{a},\,\le_D)$ is trivially true,
and so the proof of lemma \ref{lm:pp-preservation} is completed.
\end{proof}

\medskip

Now we are ready to finish the proof of theorem \ref{tm:equicons}.

\begin{proof}[Proof of Con(i) $\Rightarrow$ Con(ii)]
Assume that Shelah's Strong Hypothesis fails, i.e.
\begin{align}\notag
\text{$\pp(\mu)>\mu^+$ for some singular cardinal $\mu$}.
\end{align}
But if $\mu$ is not strong limit then there is a cardinal $\,\lambda$ such that $\cf(\mu) \le \lambda<\mu$
and $2^\lambda>\mu$. But then $2^\lambda \ge \mu^{\cf(\mu)} \ge \pp(\mu)$ as well, hence
we can apply Theorem  \ref{tm:posplus} with e.g. $\sigma=\mu$ and $\rho=\mu^+$
to obtain a $0$-dimensional Hausdorff space $X$
with $\pind(X)\le \lambda < \den(X)=\mu$ and $|X|=\Delta(X)=\mu^+$.

If $\mu$ is strong limit then we take
\begin{align*}
\text{$\lambda={(2^{\cf(\mu)})}^+$ and  $\nu=(2^{\lambda})^+$},
\end{align*}
and consider the forcing notion
\begin{align}\notag
P = \operatorname{Fn}(2^\mu\times\lambda,2;\lambda)
\end{align}
which adds $2^\mu$ Cohen subsets of $\lambda$ with conditions
of size $\le 2^{\cf(\mu)}$. Let $G$ be $P$-generic
over the ground model $V$. We claim that the generic extension $W = V[G] \supset V$ satisfies
the conditions of lemma \ref{lm:pp-preservation}.

Indeed, this follows immediately from the facts that $P$ is both $\lambda$-closed
and $\nu$-CC, using standard theorems of forcing theory. Of course, we also
have $2^\lambda = 2^\mu \ge \mu^{\cf(\mu)} \ge \pp(\mu)$ in $V[G]$, as well as
$\mu^+ < \pp(\mu)$ by lemma \ref{lm:pp-preservation}. Putting these together
we get
 \begin{align*}
 V[G]\models \cf({\mu})<{\lambda}<{\mu}<{\mu}^+<\pp({\mu})\le 2^{\lambda},
 \end{align*}
consequently, theorem \ref{tm:posplus} applied in $V[G]$ yields a
0-dimensional Hausdorff space $X$ in $V[G]$
that satisfies $|X|=\mDelta(X) = \mu^+$,
a regular cardinal, and $\pind(X)<\den(X)$.
\end{proof}

The following problem can now be raised naturally.

\begin{problem}
Is the existence of a neat (Hausdorff) space $X$ of regular size with $\pind(X) < \den(X)$
actually equivalent, and not just equiconsistent, with that of a 0-dimensional (or regular) such space?
\end{problem}

\bigskip

\section{Inequalities involving the pinning down number}

The first inequality we establish is an improvement of Pospi\v sil's classical
inequality $|X|\le 2^{2^{\den(X)}}$ for any Hausdorff space $X$.
Of course, it is only a proper improvement if the (equivalent) statements
of theorem \ref{tm:main1} fail.

\begin{theorem}\label{tm:cardless22pdX}
 $|X|\le 2^{2^{\pind(X)}}$ for every  Hausdorff space $X$.
\end{theorem}
\begin{proof}

To simplify our notation, we put $\mu=\pind(X)$  and
$\kappa=2^{2^{\mu}}$.
Let us now consider the set
\begin{align}\notag
V= \bigcup\{U\in \tau_X: |U|\le \kappa\},
 \end{align}

\begin{claim}
\label{lm:smalldeltap}
$|V|\le \kappa$.
\end{claim}

\begin{proof}[Proof of the Claim]
Assume, arguing indirectly, that $|V|> \kappa$. Then clearly
 $V$ contains an
open subspace $Y$ with $|Y|=\kappa^+$. Since $(\kappa^+)^\mu = \kappa^+$,
we may fix an enumeration $\{A_\nu:\nu<\kappa^+\}$ of $\br Y;\mu;$.
By transfinite recursion, for all $\nu <\kappa^+$ we pick
\begin{align}\notag
x_\nu \in (Y\setm \overline{A_\nu})\setm \{x_\zeta:\zeta<\nu\}.
\end{align}
This can be done because, by Pospi\v sil's theorem, $|\overline{A_\nu}|\le \kappa$,
hence $|Y\setm \overline{A_\nu}|=\kappa^+$.

Now, let $U$ be any neighborhood assignment on $Y$ such that
$U(x_\nu)=Y\setm \overline{A_\nu}$. But then $U$ can not be pinned down by a set of
size $\mu = \pind(X)$, a contradiction.
\end{proof}
Note that our aim: to show that $|X| \le \kappa$, is equivalent to showing $X = V$.

Assume, on the contrary again, that $|X|>\kappa$,
that is $X\ne V$.
Then we can define
\begin{align}\notag
 \lambda=\min\{|G|: G \in \tau_X \text{ and }  |G|>\kappa \},
\end{align}
and fix $W$, an open subset of $X$ with $|W|=\lambda$.
Of course, we also have $\pind(W)\le \pind(X) = \mu$.

Instead of the subspace topology on  $W$ inherited from $X$ we may
consider a  coarser Hausdorff  topology $\sigma$ such that the Hausdorff space $W^*=(W,\sigma)$
has weight  $\we(W^*)\le |W| = \lambda$.
Then we have $\pind(W^*) \le \pind (W) \ge \mu$ and, by Pospi\v sil's theorem, $\lambda >\kappa$ implies $\den(W^*)>\mu$.

Let $\mc B$ be a base of $W^*$ with $|\mc B|\le \lambda$ and
let $\{B_\nu:\nu<\lambda\}$ enumerate  $\mathcal{C} = \{B \in \mc B : |B| = \lambda\}$.
Note that, by the minimality of $\lambda > \kappa$,
we also have $\mathcal{C} = \{B \in \mc B : B \setm V \ne \empt \}$.

By transfinite recursion,  for all $\nu<\lambda$ we may then
pick
\begin{align}\notag
x_\nu\in B_\nu\setm \{x_\xi:\xi<\nu.\}.
\end{align}

Let $U$ be a neighborhood assignment on $X$ such that
$U(x_{\nu})=B_\nu$ for all $\nu < \lambda$.
We claim that $U$ can not be pinned down by any set of size $\mu$.
Indeed, let $A\in \br W;\mu;$.
Then
\begin{align}\notag
 |\overline{A}^\sigma|\le \kappa,
\end{align}
hence $|W\setm \overline{A}^\sigma| = \lambda$,
so $W\setm \overline{A}^\sigma\not\subs V$.
But $W\setm \overline{A}^\sigma\in \sigma$,
so there is $B\in \mc B$
such that $B\subs W\setm \overline{A}^\sigma$ and $B\not\subs V$.
Then $B \in \mathcal{C}$,
and so $B=B_\nu$ for some $\nu<\lambda$. But then $U(x_\nu)\cap A=B_\nu\cap A=\empt$,
showing that $A$ does not pin down $U$.
But this implies $\pind(W^*)>\mu$, which is a contradiction that completes the proof.
\end{proof}

\begin{theorem}\label{tm:small_or_eq_t2}
If $X$ is any Hausdorff  space which satisfies $\pind(X)<\den(X)$
then $\Delta(X)<2^{2^{\pind(X)}}$.
\end{theorem}

\begin{proof}
Since $|X|\le 2^{2^{\pind(X)}}$ by Theorem \ref{tm:cardless22pdX},
$\Delta(X)\ge 2^{2^{\pind(X)}}$ would imply
$|X|=\Delta(X)=2^{2^{\pind(X)}} = \kappa$.
But for $\mu=\pind(X)$ we have  $\kappa^\mu=\kappa$,
hence we can apply  lemma \ref{lm:cfkappalambda1} with $\lambda = \mu^+ \le \den(X)$
to conclude that $\pind(X) = \mu \ge \lambda$, which contradicts our choice
of $\mu$ and $\lambda$. Thus we must have $\Delta(X)<2^{2^{\pind(X)}}$.
\end{proof}

This is all the inequalities we have for Hausdorff spaces and now we turn to the
study of regular spaces. Perhaps the best known and most frequently applied inequality
concerning a regular space $X$ that involves the density is $\we(X) \le 2^{\den(X)}$.
This led us to raise the following question.

\begin{problem}\label{pr:we}
Does $\we(X)\le 2^{\pind(X)}$ hold for every regular space $X$?
\end{problem}

This question remains wide open but we managed to obtain quite a few
interesting and non-trivial results abut the cardinal function $\pind(X)$
for regular $X$.

We recall that a topological space $X$ is called {\em weakly separated} iff there is
a neighborhood assignment $U$ on $X$ such that either $x\not\in U(y)$ or $y\notin U(x)$
whenever $\{x,y\}\in \br X;2;$.
The related cardinal function $\operatorname{R}(X)$ is defined as the supremum of the
cardinalities of all weakly separated subspaces of $X$. Since $\operatorname{R}(X) \le \we(X)$
but ``not much less than'' $\we(X)$, our following result may be considered as a
partial affirmative answer to problem \ref{pr:we}.

\begin{lemma}\label{lm:Rle2pd}
If $X$ is a neat regular space then
$\operatorname{R}(X)\le 2^{\pind(X)}$.
\end{lemma}

\begin{proof}
Let $Y$ be any weakly separated subspace of $X$;
we want to show that $|Y|\le 2^{\pind(X)}$. It is easy to see that
we can find a  coarser regular  topology $\sigma$ on $X$
such that for the space
 $X^*=(X,\sigma)$ we have
 $\we(X^*)\le |X|$ and $Y$ remains weakly separated in $X^*$.

Clearly, $X^*$ is also neat, hence $$\pi (X^*) \le\we(X^*)\le |X^*| = \Delta(X^*)$$ imply $\den(X^*)=\pind(X^*)$
by Lemma \ref{lm:piDelta}.   Since $\pind(X^*)\le \pind(X)$,
we may then conclude
\begin{align}\notag
|Y|\le \we(X^*)\le 2^{\den(X^*)}=2^{\pind(X^*)}\le 2^{\pind(X)}.
\end{align}
\end{proof}

We do not know if the neatness condition is necessary in the previous result but
it is not needed in the next one.

\begin{lemma}\label{lm:dle2pd}
$\den(X)\le 2^{\pind(X)}$ holds for any regular space $X$.
\end{lemma}

\begin{proof}
Let $\mc H$ be a maximal disjoint family of pairwise disjoint neat open subspaces of $X$.
Then $|\mc H|\le c(X)\le \pind(X)$, moreover $\bigcup\mc H$
is dense in $X$.   We have $\den(H)\le \operatorname{R}(H)\le 2^{\pind(H)}$ for all $H\in \mc H$
by Lemma \ref{lm:Rle2pd},
consequently
\begin{align}\notag
\den(X) = \den\big(\bigcup \mc H\big) = \sum_{H\in \mc H}\den(H)
\le |\mc H|\cdot 2^{\pind(X)}=2^{\pind(X)}.
\end{align}
\end{proof}

Our following result does not involve the pinning down number, still
it will be crucial in our later results that do.

\begin{theorem}\label{tm:project}
 Let $X$ be a regular space and
$\mu$ be a regular cardinal such that
$$\operatorname{hL}(X)\le \mu\le \min(\Delta(X),\we(X)).$$
Then there is  a regular continuous image $Y$ of $X$ for which
$\Delta(Y)\ge \we(Y)=\mu$ holds.
\end{theorem}

\begin{proof}
For every open set $U\subs X$ we let
\begin{align}\notag
 \mc G_U=\{V\in \tau_X: \overline V\subs U\}.
\end{align}
Since $X$ is regular we have $\bigcup\mc G_U=U $, and $\operatorname{hL}(X)\le {\mu}$ implies that
we can fix $\mc H_U\in \br \mc G_U;\le {\mu};$ with
$\bigcup\mc H_U=U $.

Let $\mc M$ be an elementary submodel of size ${\mu}$ of $H_\vartheta$ for a large enough
regular cardinal $\vartheta$
such that everything relevant belongs to $ \mc M$, ${\mu}+1\subs \mc M$, and $\mc M$ is $< \mu$-covering, i.e.
for each $B\in \br \mc M;<\mu;$
there is $C\in \br \mc M;<\mu;\cap \mc M$ with $B\subs C$.

For $x,y\in X$ let us put
\begin{align}\notag
 x\sim y\ \text{ iff }\ \forall U\in \mc M\cap \tau_X\ (x\in U
\Longleftrightarrow y\in U).
\end{align}
Then $\sim $ is clearly an equivalence relation on $X$.

\begin{claim}\label{cl:equiv}
If $x\not\sim y$ then there are disjoint open sets $U_x, \,U_y \in \mc M \cap \tau_X$
such that $x\in U_x$ and $y\in  U_y$.
\end{claim}
\begin{proof}[Proof of the Claim]
Assume that $U\in \mc M\cap \tau_X$ is such that $x\in U$ and
$y\notin U$. Then we have $\mc H_U\subs \mc M$
because $\mc H_U\in \mc M$ and  $|\mc H_U|\le \mu$.
We have $x\in V$ for some $V\in \mc M\cap \mc H_U$ and clearly $y \notin \overline V$
because $\overline{V}\subs U$.
Thus  $U_x = V$ and $U_y = X\setm \overline V$ are as required.
\end{proof}

Let $[x]$ denote the $\sim $-equivalence class of $x \in X$.
Using Claim \ref{cl:equiv} we can see then that
\begin{align}\label{eq:[x]}
 [x]=\bigcap\{U\in \mc M\cap \tau_X: x\in U\}.
\end{align}
It follows that if $U\in \mc M \cap \tau_X$ then
\begin{align}\label{eq:G}
 U=\bigcup\{[x]: x\in U\}.
\end{align}
Also, for every point $x \in X \cap \mc M$ we have $[x] = \{x\}$
because $\operatorname{hL}(X)\le \mu$ implies $\psi(x,X) \le \mu$.

Let us put
\begin{align}\notag
 Y=X/\sim\,\,=\,\{[x]:x\in X\}
\end{align}
and
\begin{align}\notag
 \mc B=\{U/\sim\,\,:\, U\in \mc M\cap \tau_X\}.
\end{align}
$\mc B$ is well-defined by (\ref{eq:G}) and it is clearly closed under finite intersections,
hence it is the base of a topology $\sigma$ on $Y$.
That this topology $\sigma$ is Hausdorff is immediate from claim \ref{cl:equiv}.
But it is also regular: Indeed, if $[x] \in U/\sim\,$ with $U\in \mc M\cap \tau_X$
then, as we have seen, there is $V\in \mc M\cap \mc H_U$ with $x \in V$. Now, it is
easy to see that then $[x] \in V/\sim\, \subs {\overline{V/\sim}\,}^\sigma \subs U/\sim\,$.

Let us next define the map $\varphi:X\to Y$ by the formula
\begin{align}\notag
 \varphi(x)=[x].
\end{align}
Then $\varphi$ is obviously a continuous surjection, hence $Y$
is  a regular continuous image of $X$.

\begin{claim}
$\Delta(Y)\ge {\mu}$.
\end{claim}

\begin{proof}
Let $U\in \mc M\cap \tau_X$ be non-empty.
Then${\mu}+1\subs \mc M$ and $|U|\ge{\mu}$ imply
$|U\cap \mc M|={\mu}$. But for every $x \in U\cap \mc M$ we have $[x] = \{x\}$,
hence
\begin{align}\notag
 |U/\sim|\ge |U\cap \mc M|={\mu},
\end{align}
completing the proof.
\end{proof}

\begin{claim}
$\we(Y)={\mu}$.
\end{claim}

\begin{proof}
Clearly  $\we(Y)\le |\mc B|={\mu}$.
Next, as $\mc M$ is $< \mu$-covering, for any $\mathcal{G} \in [\mc M \cap \tau_X]^{< \mu}$
there is $\mathcal{H} \in [\mc M \cap \tau_X]^{< \mu} \cap \mc M$
with $\mc G\subs \mc H$.  Then $\mc H$ is not a base of $X$ because
$\we(X) \ge \mu$,
so there are a point $x\in X$ and an open set $V$ containing $x$
such that for every $H\in \mc H$ with $x\in H$ we have $H\setm V \ne \empt$.
By elementarity we can then find to $\mathcal{H}$ such witnesses
$x$ and $V$ in $\mc M$ as well. But then
for each $H\in \mc H$ with $x\in H$ there is $y\in (H\setm V) \cap \mc M$,
hence $[y]\in (H/\sim)\setm (V/\sim)$.

This shows that
$\{H/\sim\,:\,H\in \mc H\}$ and consequently $\{G/\sim\,:\,G\in \mc G\}$
is not a base of $\sigma$. Since every member of $[\mc B]^{< \mu}$
is of the form $\{G/\sim\,:\,G\in \mc G\}$ for some $\mathcal{G} \in [\mc M \cap \tau_X]^{< \mu}$,
we conclude that no member of $[\mc B]^{< \mu}$ is a base for $\sigma$.
This implies $\we(Y) = \mu$ because it is known that any base of any space
has a subset which is a base and has cardinality equal to the weight of the space.
\end{proof}

This completes the proof of theorem \ref{tm:project}.
\end{proof}

We note that if the space $X$ in theorem \ref{tm:project} is assumed to be Tychonov
rather than regular then its continuous image $Y$ can also be chosen to be Tychonov.
In fact, in that case the proof is significantly simpler.

The following result gets pretty close to the affirmative solution of problem
\ref{pr:we}.

\begin{theorem}\label{tm:mindeltaw}
If $X$ is any regular space then
\begin{align}\notag
\min\{\Delta(X), \we(X)\}\le 2^{\pind(X)}.
\end{align}
\end{theorem}

\begin{proof}
Our proof is indirect, so we assume that
\begin{align}\notag
\min\{\Delta(X), \we(X)\}>2^{\pind(X)}.
\end{align}
Then from $2^{\pind(X)} < \we(X)\le 2^{\den(X)}$, we get
\begin{align}\label{ine:pindden}
\pind(X)<\den(X).
\end{align}

Let us consider the family
\begin{align}\notag
\mc G=\{G\in \tau_X: \we(G)\le \Delta(G)\},
\end{align}
then for all $G\in \mc G$ we have
$\den(G)=\pind(G)$ by Lemma \ref{lm:piDelta}.
If $\mc H\subs \mc G$ is a maximal disjoint subfamily of $\mc G$
then we have
\begin{align}\notag
 \den(\overline{\bigcup\mc G})=
\den(\overline{\bigcup\mc H}) =
\den({\bigcup\mc H})
= \sum_{H\in \mc H}\den(H)\le \operatorname{c}(X)\cdot\pind(X)=\pind(X).
\end{align}
But then $\pind(X)<\den(X)$ implies $X\setm \overline{\bigcup \mc G}\ne \empt$,
hence we may choose a {\em neat} non-empty open subset $G\subs X\setm \overline{\bigcup \mc G}$.
Then we have
\begin{align}
\we(G) > \Delta(G)\ge \Delta(X)>2^{\pind(X)}\ge
2^{\pind(G)}.
\end{align}

Since $G$ is regular and neat, we may apply lemma \ref{lm:Rle2pd}
to conclude that $\operatorname{hL}(G)\le \operatorname{R}(G)\le 2^{\pind(G)}$.
Thus we may apply theorem \ref{tm:project} to $G$ with $\mu=(2^{\pind(G)})^+$
to obtain a regular continuous image $Y$ of $G$ such that
$\Delta(Y)\ge \we(Y)= \mu=(2^{\pind(G)})^+$. But then,
by lemma \ref{lm:piDelta}, we have $\pind(Y)=\den(Y)$ and hence $\we(Y)\le 2^{\den(Y)}=2^{\pind(Y)}$.
Since $Y$ is a continuous image of $G$, by proposition \ref{prop:im} we also have
$\pind(Y)\le \pind(G)$. So on one hand we have $\we(Y) \le 2^{\pind(G)}$,
while on the other hand $\we(Y) = (2^{\pind(G)})^+$.
This blatant contradiction completes our proof.
\end{proof}

Now we can present a strengthened version of lemma \ref{lm:dle2pd}.

\begin{theorem}\label{tm:dless2pd}
For every regular space $X$ we have $\den(X)<2^{\pind(X)}$.
\end{theorem}

\begin{proof}
Assume, on the contrary, that $\den(X)\ge 2^{\pind(X)}$.
Then, by Lemma \ref{lm:dle2pd},
we actually have $\den(X)=2^{\pind(X)}$.

Let us put
\begin{align}\notag
\mc G=\{G\in \tau_X: \den(G)<2^{\pind(X)}\}
\end{align}
and $\mc H\subs \mc G$ be a maximal disjoint subfamily.
Then we have
\begin{align}\notag
 \den(\overline{\bigcup\mc G})=
\den(\overline{\bigcup\mc H}) =
\den({\bigcup\mc H})
= \sum_{H\in \mc H}\den(H) < 2^{\pind(X)} = \den(X)
\end{align}
because $|\mc H| \le \operatorname{c}(X)\le \pind(X)$
and $\cf(2^{\pind(X)})>\pind(X)$.

Thus $X\setm \overline{\bigcup \mc G}\ne \empt$
because $\den(X)=2^{\pind(X)}$, and so it has
a neat non-empty open subset $G$. Clearly, then $\den(G)=2^{\pind(X)}$,
hence $|G| = \Delta(G) \ge 2^{\pind(X)}$.
But $\Delta(G)>2^{\pind(X)} \ge 2^{\pind(G)}$ would imply $$\we(G)\le 2^{\pind(G)} \le 2^{\pind(X)} \le \Delta(G)$$
by Theorem \ref{tm:mindeltaw},  hence
$\den(G)=\pind(G)$ by lemma \ref{lm:piDelta},
which clearly contradicts $\den(G)=2^{\pind(X)}$.
Consequently, we have $|G| = \Delta(G)=2^{\pind(X)}$.

Because of $(2^{\pind(X)})^{\pind(X)} =2^{\pind(X)}$, however,
we can apply  lemma \ref{lm:cfkappalambda1} to the neat space $G$
with $\kappa=2^{\pind(X)}$ and $\lambda=\pind(X)^+$
to conclude that $\pind(G) \ge \lambda =\pind(X)^+$, which is again a contradiction.
\end{proof}

Our final result may be considered as the analogue of theorem \ref{tm:dless2pd} for regular
rather than just Hausdorff spaces.

\begin{theorem}\label{tm:small_or_eq}
If $X$ is any regular space such that $\pind(X)<\den(X)$
then  $\Delta(X)<2^{\pind(X)}$.
\end{theorem}

\begin{proof}
We prove the contrapositive of this statement:
Assume that $X$ is regular and $\Delta(X)\ge 2^{\pind(X)}$.
Then for any non-empty open subset of $G \subs X$ we have
$$\Delta(G)\ge \Delta(X)\ge 2^{\pind(X)}\ge 2^{\pind(G)}.$$

Now, if $\Delta(G)> 2^{\pind(G)}$ then we have $\we(G)\le 2^{\pind(G)}$
by Theorem \ref{tm:mindeltaw}, and so $\we(G)<\Delta(G)$
which implies $\pind(G)=\den(G)$ by Lemma \ref{lm:piDelta}.

Otherwise  $\Delta(G)= 2^{\pind(G)}$, hence if $G$ is also neat then, as above,
we can apply  lemma \ref{lm:cfkappalambda1} for $G$
with $\kappa=2^{\pind(G)}$ and $\lambda=\pind(G)^+$ to conclude that $\pind(G) = \den(G)$.

This, of course, implies that $\pind(G) = \den(G)$ holds for all neat open $G \subs X$,
consequently $\pind(X) = \den(X)$ by lemma \ref{lm:ex2cardhomex1}.
\end{proof}

\end{document}